\documentclass[reqno]{amsart}

\usepackage{amssymb,amsthm,amsmath} 
\usepackage[top=25truemm,bottom=20truemm,left=20truemm,right=20truemm]{geometry}
\usepackage{graphicx}
\usepackage{latexsym}
\usepackage{comment}
\usepackage{setspace}

\numberwithin{equation}{section}
\theoremstyle{plain}
\newtheorem{theorem}{Theorem}[section]
\newtheorem{lemma}[theorem]{Lemma}
\newtheorem{proposition}[theorem]{Proposition}

\theoremstyle{definition}
\newtheorem{definition}[theorem]{Definition}
\newtheorem{remark}[theorem]{Remark}

\begin{document}



\title[Nonlocal NLS around standing wave solutions]
{Large-time existence results for the nonlocal NLS around ground state solutions}

\author[Hideo Takaoka, Toshihiro Tamaki]{Hideo Takaoka and Toshihiro Tamaki}

\address{Hideo Takaoka \newline
\indent Department of Mathematics, Kobe University, Kobe, 657-8501, Japan}
\email{takaoka@math.kobe-u.ac.jp}

\address{Toshihiro Tamaki \newline
\indent Department of Mathematics, Kobe University, Kobe, 657-8501, Japan\newline
\indent Structural Analysis Section, C.A.E. Technical Department, Kawasaki Technology Co. Ltd.,\newline
\indent 3-1, Kawasaki-cho, Akashi, 673-0014, Japan}
\email{ttamaki@math.kobe-u.ac.jp}


\begin{abstract}
This paper discusses about solutions of the nonlocal nonlinear Schr\"odinger equation.
We prove that the solution remains close to the orbit of the soliton for a large-time, if the initial data is close to the ground state solitons. 
The proof uses the hyperbolic dynamics near ground state, which exhibits properties of local structural stability of solutions with respect to the flows of the nonlocal nonlinear Schr\"odinger equation.
\end{abstract}

\maketitle

\section{Introduction}\label{sec:intro}

In this paper, we consider the Cauchy problem for the following focusing nonlocal nonlinear Schr\"odinger equation
\begin{equation}\label{eq:nnls}
\begin{gathered}
i\partial_t u(t,x)-\partial_x^2 u(t,x)=u(t,x)^2u^{\star}(t,x),\quad (t,x)\in\mathbb{R}^2,\\
u(0,x)=u_0(x),\quad x\in\mathbb{R},
\end{gathered}
\end{equation}
where $\star$ indicates complex conjugate with the reflection in the $x$-axis, namely $f^{\star}(x)=\overline{f(-x)}$.
This equation was first introduced by Ablowitz and Musslimani \cite{ablowitz2} as a new nonlocal reduction of the Ablowitz-Kaup-Newell-Segur (AKNS) system \cite{ablowitz1}.
The physical aspect of model and their effect on expected metamagnetic structures are founded in \cite{gadzhimuradov}, and also in \cite{gurses,yang}.

The equation \eqref{eq:nnls} has an infinite number of conserved quantities.
For example, the quasipower
\begin{equation}\label{eq:mass}
M[u](t)=\frac12\int_{\mathbb{R}}u(t,x)u^{\star}(t,x)\,dx,
\end{equation}
and the Hamiltonian
\begin{equation*}
H[u](t)=\frac12\int_{\mathbb{R}}\partial_x(u(t,x))\partial_x(u^{\star}(t,x))\,dx-\frac{1}{4}\int_{\mathbb{R}}u(t,x)^2u^{\star}(t,x)^2\,dx
\end{equation*}
are conserved for solutions to \eqref{eq:nnls}.
Moreover, the equation \eqref{eq:nnls} is known as a Hamiltonian dynamical system under the symplectic form
\begin{equation*}
\omega(f,g)=\Im\int_{\mathbb{R}}f(x)g^{\star}(x)\,dx~\left(=-\Re\int_{\mathbb{R}}if(x)g^{\star}(x)\,dx\right),
\end{equation*}
namely
\begin{equation*}
\omega\left(\frac{\partial u}{\partial t},w\right)=d_uH[w],
\end{equation*}
where
\begin{equation*}
d_uH[w]=\left.\frac{d}{ds}\right|_{s=0}H[u+sw]=\Re\int_{\mathbb{R}}\left(-\partial_x^2u-u^2u^{\star}\right)w^{\star}\,dx.
\end{equation*}
Like for the local nonlinear Schr\"odinger equation
\begin{equation}\label{eq:nls}
i\partial_t u(t,x)-\partial_x^2 u(t,x)=u(t,x)^2\overline{u(t,x)},\quad (t,x)\in\mathbb{R}^2,
\end{equation}
the equation \eqref{eq:nnls} is PT (parity-time) symmetric, namely, the equation \eqref{eq:nnls} is invariant under the transform of $(t,x,u)\to (-t,-x,u^*)$.
In contrast to the equation \eqref{eq:nnls}, the equation \eqref{eq:nls} is expressed as a Hamiltonian flow of the form endowed by the standard symplectic form
\begin{equation*}
\omega_{\mathrm{NLS}}\left(\frac{\partial u}{\partial t},w\right)=d_uH_{\mathrm{NLS}}[w],
\end{equation*}
where
\begin{equation*}
H_{\mathrm{NLS}}[u](t)=\frac12\int_{\mathbb{R}}\partial_xu(t,x)\partial_x\overline{u(t,x)}\,dx-\frac{1}{4}\int_{\mathbb{R}}u(t,x)^2\overline{u(t,x)}^2\,dx
\end{equation*}
and
\begin{equation}\label{eq:Omega}
\omega_{\mathrm{NLS}}(f,g)=\Im\int_{\mathbb{R}}f(x)\overline{g(x)}\,dx~\left(=-\Re\int_{\mathbb{R}}if(x)\overline{g(x)}\,dx\right).
\end{equation}
Both equations \eqref{eq:nnls} and \eqref{eq:nls} are known as a completely integrable model, which have a Lax pair and an infinite number of  conservation laws.
Then the inverse scattering transform was applied to construct a variety of solutions for the case of rapidly decaying initial data \cite{ablowitz3,zakharov}.

Without the phase modulation term $e^{i\theta} ~(\theta\in\mathbb{R})$, equations \eqref{eq:nnls} and \eqref{eq:nls} have a family of one-parameter family of soliton solutions with smooth and rapidly decaying at infinity
\begin{equation}\label{eq:standing}
e^{-it\alpha^2}Q_{\alpha}(x)\quad (\alpha>0),
\end{equation}
where $Q_{\alpha}(x)=2\sqrt{2}\alpha/(e^{\alpha x}+e^{-{\alpha x}})$.
We emphasize that the equation \eqref{eq:nnls} is an interesting one possessing  a two-parameter family of solitons
\begin{equation}\label{eq:standing-2}
u_{\alpha,\beta}(t,x)=\frac{\sqrt{2}(\alpha+\beta)}{e^{i\alpha^2 t+\alpha x}+e^{i\beta^2 t-\beta x}} \quad (\alpha,\beta>0),
\end{equation}
which differs from the local nonlinear Schr\"odinger equation \eqref{eq:nls}.
The soliton solutions \eqref{eq:standing} are represented by the choice of  the parameter $\alpha=\beta$ in \eqref{eq:standing-2}, namely $u_{\alpha,\alpha}(t,x)=e^{-it\alpha^2}Q_{\alpha}(x)$.

It is known that the Cauchy problem to \eqref{eq:nls} is local well-posed in $H^s$ for $s\ge0$ (e.g., see \cite{cazenave,tsutsumi}).
In a similar way with some minor changes, the local well-posedness for \eqref{eq:nnls} holds in the same space.
In \cite{genoud}, Genoud proved the existence of finite-time blow-up soliton solutions to \eqref{eq:nnls} of the form \eqref{eq:standing-2} with arbitrarily small initial in $H^s$ for $s\ge 0$, though the Cauchy problem for the local nonlinear Schr\"odinger equation \eqref{eq:nls} is globally well-posed in $H^s$ for $s\ge 0$. 
Moreover, it is proved in \cite{genoud} that the soliton solutions of the form \eqref{eq:standing} are unstable by blowing up singularities near the origin for solution \eqref{eq:standing-2}. 
What is remarkable is that this type of strong blow-up instability of soliton solutions can not be observed for the local nonlinear Schr\"odinger equation \eqref{eq:nls}.
A suitable concept for continuation of weak solutions beyond possible blow-up was observed by Rybalko \cite{rybalko}.

We refer to the papers \cite{chen,zhao} for global existence results.
In \cite{chen}, Chen, Liu and Wang proved the global existence and uniqueness of the solutions for data belonging to some super-critical function spaces $E^{s}_{\sigma}$.
Recently, the existence of global solutions in the weighted Sobolev space $H^{1,1}$ with the smallness assumption on the $L^1$ norm was obtained by Zhao and Fan \cite{zhao}, based upon the inverse scattering theory.
One addresses with further remarks on global existence results.
In \cite{okamoto}, Okamoto and Uriya considered the final state problem for the nonlocal nonlinear Schr\"odinger equations in which the nonlinear part of \eqref{eq:nnls} is replaced by dissipative nonlinearity of the form $\lambda u(t,x)^2u^{\star}(t,x)$ for $\lambda\in\mathbb{C},~\Im\lambda\ne 0$.
They proved the asymptotic behavior of solutions as time goes to infinity, by observing that the asymptotic profile may depend on the solutions.  

Let us take $u_{\alpha,\alpha}(t,x)$ in \eqref{eq:standing} and $u_{\alpha,\beta}(t,x)$ in \eqref{eq:standing-2}, for purpose of studying the large time existence of soliton that is close to the soliton solutions for initial data $u_0$, not small, but close to $Q_{\alpha}$.
If  one chooses $\alpha>0$ and $\beta>0$ in \eqref{eq:standing} and \eqref{eq:standing-2} such that $|\alpha-\beta|>0$ is small enough, then
\begin{equation*}
\|u_{\alpha,\alpha}(0,\cdot)-u_{\alpha,\beta}(0,\cdot)\|_{H^k}\sim C_k|\alpha-\beta|
\end{equation*}
for any $k\in \mathbb{N}\cup\{0\}$.
As pointed out in \cite[Theorem 2]{genoud}, $u_{\alpha,\beta}(t)$ blows up in $L^{\infty}$ as $|t|\to T_{\alpha,\beta}$, more precisely,
\begin{equation*}
\lim_{t\to \pi/(\alpha^2-\beta^2)}u_{\alpha,\beta}(t,0)=\infty,
\end{equation*}
where $T_{\alpha,\beta}=\pi/(\alpha^2-\beta^2)$.
In particular, the solution $u_{\alpha,\beta}(t)$ blows up in $H^k$ at $|t|=T_{\alpha,\beta}$, that is
\begin{equation}\label{eq:blowup}
\lim_{t\to T_{\alpha,\beta}}\|u_{\alpha,\alpha}(t,\cdot)-u_{\alpha,\beta}(t,\cdot)\|_{H^k}=\infty,\quad \lim_{t\to T_{\alpha,\beta}}\|u_{\alpha,\beta}(t,\cdot)\|_{H^k}=\infty.
\end{equation}
At the same time, it should be emphasized that
\begin{equation}\label{eq:blowup2}
\inf_{\theta\in\mathbb{R}}\|u_{\alpha,\alpha}(t,\cdot)-e^{i\theta}u_{\alpha,\beta}(t,\cdot)\|_{H^k}\ge C_k |\alpha-\beta||t|
\end{equation}
holds at least as long as $t$ stays in $1\lesssim |t|\ll 1/|\alpha-\beta|$.
The estimate \eqref{eq:blowup2} implies the orbital instability of $u_{\alpha,\alpha}$.

The purpose of this paper is to analyze the details in the dynamics of the solution near the solitons, which is somewhat striking analogy on the soliton solutions observed.
We investigate the large time stability of solitons $u_{\alpha,\alpha}$ in a neighborhood that does not contain any $u_{\alpha,\beta}~(\alpha\ne\beta)$. 
As far as we know, there is no large time stability of $u_{1,1}$ to \eqref{eq:nnls} by approach as a complete integral model.

For simplicity, and without loss of generality, we shall consider the large time stability of solitons $u_{1,1}$ for initial data $u_0$ close to $Q_{1}$.

Before stating our main result, we first introduce the following space of initial data.

\begin{definition}\label{def:space}
We set $Q=Q_1,~Q'_{\alpha}=\partial_{\alpha}Q_{\alpha}$ and $Q'=Q'_{\alpha}|_{\alpha=1}$, which lead $Q'=(1+x\partial_x)Q$.

For $\varepsilon>0$, we let the function space $\mathcal{S}_{\varepsilon}$ of $H^1$ be an initial data set such that for $f\in H^1$,  $f$ belongs to space $\mathcal{S}_{\varepsilon}$ if and only if
\begin{equation}\label{eq:f}
\left\|f-\left(Q+A_e iQ+B_eQ'+A_oi\partial_xQ\right)\right\|_{H^1}\le \varepsilon^2,
\end{equation}
for $A_e,~B_e,~A_o\in\mathbb{R}$ such that $|A_e|+|B_e|\le |A_o|=\varepsilon$.
\end{definition}
\begin{remark}\label{rem:space}
If we consider the proper symplectically orthogonal decomposition based on Proposition $\ref{prop:eighenvalue}$ which will be described later, then the function $f$ in \eqref{eq:f} satisfies
$$
f(x)=Q(x)+a_e iQ(x)+b_eQ'(x)+a_oi\partial_xQ(x)+b_oxQ(x)+g(x),
$$
for $a_e,~b_e,~a_o,~b_o\in\mathbb{R}$ and $g\in H^1$ such that $|a_e|+|b_e|\lesssim \varepsilon,~\ |a_o|\sim\varepsilon,~|b_o|+\|g\|_{H^1}\lesssim\varepsilon^2$ and
$$
\Re\int_{\mathbb{R}}g(x)Q(x)\,dx=0,\quad \Im\int_{\mathbb{R}}g(x)Q'(x)\,dx=0,
$$
$$
\Re\int_{\mathbb{R}}g(x)\partial_xQ(x)\,dx=0,\quad \Im\int_{\mathbb{R}}g(x)xQ(x)\,dx=0.
$$
Moreover, we easily see that
\begin{equation}\label{eq:data-ex}
u_{1,\beta}(0,\cdot)\not\in \mathcal{S}_{\varepsilon},
\end{equation}
if $|1-\beta|\sim \varepsilon\ll 1$.
In addition,
\begin{equation}\label{eq:data-diff}
d(f,Q)\sim \varepsilon
\end{equation}
for $f\in \mathcal{S}_{\varepsilon}$ provided $\varepsilon\ll 1$, where $d(f,Q)$ is the distance function from the grand state $Q$ as
\begin{equation*}
d(f,Q)=\inf_{\theta\in\mathbb{R}}\|f-e^{i\theta}Q\|_{H^1}.
\end{equation*}
We will prove \eqref{eq:data-ex} and \eqref{eq:data-diff}  in Appendix.
\end{remark}

The main result of the paper is the following, which hopefully clarifies the differences in \eqref{eq:blowup2}.

\begin{theorem}\label{thm:main}
There exists $\varepsilon_0>0$ such that for any $0< \varepsilon\le \varepsilon_0$, if $u_0\in\mathcal{S}_{\varepsilon}$, then there exists a unique solution $u(t)$ to \eqref{eq:nnls} on $[-T_{\varepsilon},T_{\varepsilon}]$, where $T_{\varepsilon}=c \log(1/\varepsilon)$ with $0<c\ll 1$ independent of $\varepsilon$.
Moreover, the solution $u(t)$ satisfies
\begin{equation}\label{eq:stab}
\sup_{|t|\le T_{\varepsilon}}d(u(t),Q)\lesssim \varepsilon.
\end{equation}
\end{theorem}

\begin{remark}
In Theorem $\ref{thm:main}$, the large time stability of solitons $u_{1,1}$ is achieved by assuming the initial data in $\mathcal{S}_{\varepsilon}$.
We do not know whether the function $\mathcal{S}_{\varepsilon}$ of initial data is the best setting for the large time stability of such solitons.
\end{remark}

\begin{remark}
The explicit formulae of solution $u(t)$ obtained in Theorem $\ref{thm:main}$ is founded in \eqref{eq:explicit} 
described later.
\end{remark}

\begin{remark}
By taking $\alpha=1$ in \eqref{eq:blowup2} and putting $\varepsilon=|1-\beta|$, we have
\begin{equation}\label{eq:ccc}
d(u_{1,\beta}(t),Q)\ge C_1\varepsilon  |t|
\end{equation}
for $1\lesssim |t|\ll 1/\varepsilon$. 
As seen in Theorem $\ref{thm:main}$, the solution might deviate from the orbit of $u_{1,1}$.
However, it is worth to remark that solution stay in the $\varepsilon$-neighborhood of the orbit of $u_{1,1}$ at least until $T_{\varepsilon}$, comparing the corresponding order of deviation time for $u_{1,\beta}$ with the size of $d(u_{1,\beta},Q)$ in \eqref{eq:ccc}.
\end{remark}

In Theorem $\ref{thm:main}$, we restrict ourselves to $H^1$ solutions.
One can lower the regularity than $s=1$ in the theorem.
We do not pursue this direction in this paper. 

In this paper, we focus on the study of large time existence of solutions to the nonlinear Schr\"odinger equation with the nonlocal cubic nonlinearity, where the global existence of solutions with small initial data was constructed by the inverse scattering transform method.
Incidentally, the cubic nonlinearity is crucial in one space dimensional in view of a description of the asymptotic behavior of solutions.
The proof of Theorem \ref{thm:main} does not use the linear stability of global solution $u_{\alpha,\alpha}$ to \eqref{eq:nnls}.
We do not know whether there are other nonlinearities having a physical meaning that the argument in the proof of Theorem \ref{thm:main} is applied.

The paper is organized as follows.
In Section \ref{sec:notation}, we collect some of the definitions and notations that we use throughout the paper.
In Section \ref{sec:preliminary}, some preliminary formulations of the equations are described.
We consider small fluctuations around the ground state $Q$.
We obtain the symplectic decompositions both for even and odd function induced by two linearized operator around the ground state.
Section \ref{sec:proof} is devoted to the proof of Theorem \ref{thm:main}.  
In Appendix \ref{secA1}, we explain the calculation of the estimates in \eqref{eq:blowup} and \eqref{eq:blowup2}.

\section{Notation}\label{sec:notation}

For positive real number $a$ and $b$, the notion $a \lesssim b$ means that there is a constant $c$ such that $a \le cb$.
When $a \lesssim b$ and $b \lesssim a$, we write $a \sim b$.
The notion $a\ll b$ stands for $a\le cb$ for small constant $c>0$.
Moreover, $b+$ means $b+\varepsilon$ for $0<\varepsilon\ll 1$ small enough.
Similarly, $b-$ means $b-\varepsilon$ for $0<\varepsilon\ll1$ small enough.

Throughout the paper, $c$ and $C$ denote various constants the value of which may change from line to line. 

For $s \in \mathbb{R}$, we define the Sobolev spaces $H^s=H^s(\mathbb{R})$ equipped with the norm
\begin{equation*}
\| f \|_{H^s} = \left(\int_{\mathbb{R}} \langle \xi\rangle^{2s} |\widehat{f}(\xi)|^2 \,d\xi\right)^{1/2},
\end{equation*}
where  $\langle x \rangle = 1 + |x|$.
We also use $L^2=H^0$.
The $L^2$ space is endowed with the norm associated to the inner product 
\begin{equation*}
(u,v)=\int_{\mathbb{R}}u(x)\overline{v(x)}\,dx.
\end{equation*}
Let us also use the semi-inner product
\begin{equation*}
\langle u\mid v\rangle=\Re(u,v).
\end{equation*}
Following the notation in papers \cite{chang,weinstein1}, we make the following definitions
\begin{equation*}
\mathcal{L}_{\pm}v =   -\partial_x^2v+v-2Q^2v\pm Q^2\overline{v},
\end{equation*}
which is non self-adjoint operator.
We make the following definitions
\begin{equation*}
L_{+} =   -\partial_x^2+1-3Q^2,\quad L_{-} =   -\partial_x^2+1-Q^2.
\end{equation*}
Moreover, $\|W\|_{L^p}=\|u\|_{L^p}+\|v\|_{L^p}$ for $W=(u,v)^t\in L^p(\mathbb{R})\times L^p(\mathbb{R})$.

Let $\sigma_1,~\sigma_2, ~\sigma_3$ be the Pauli matrices such that
\begin{equation*}
\sigma_1=
 \begin{pmatrix}
     0 & 1 \\
     1 & 0
  \end{pmatrix},
  \quad
  \sigma_2=i
 \begin{pmatrix}
    0 & -1 \\
     1 & 0
  \end{pmatrix},
  \quad
\sigma_3=
 \begin{pmatrix}
     1 & 0 \\
     0 & -1
  \end{pmatrix},
\end{equation*}
which satisfy $\sigma_j^2=I$.

The dot notation refers to a time derivative such as $\dot{\theta}=\mathrm{d}\theta/\mathrm{d}t$.

We recall some previous results on the grand state $Q_{\alpha}$ in \eqref{eq:standing}. 
The grand state $Q_{\alpha}$ satisfies the equation
\begin{equation*}
-\partial_x^2Q_{\alpha}+\alpha^2 Q_{\alpha}=Q_{\alpha}^3.
\end{equation*}
$Q'=\partial_{\alpha}Q_{\alpha}|_{\alpha=1}=(1+x\partial_x)Q$ defined in Definition \ref{def:space} satisfies
\begin{equation}\label{eq:Q'}
(-\partial_x^2+1-3Q^2)Q'=-2Q.
\end{equation}

\section{Preliminary}\label{sec:preliminary}

\subsection{Reformulate the setup of the problem}

In this subsection, we try to reformulate the Cauchy problem \eqref{eq:nnls} as the perturbed initial value problem near soliton solutions by finding the formula
\begin{equation}\label{eq:explicit}
u(t,x)=e^{i\theta(t)}\left(Q_{\alpha(t)}(x)+v(t,x)\right).
\end{equation}
We begin with local well-posedness.
It is easy to check from the results in \cite{cazenave,tsutsumi} that the time local solution of the form \eqref{eq:explicit} exists.
We will state the following local well-posedness result without proof.

\begin{proposition}[\cite{cazenave,tsutsumi}]\label{prop:LWP}
The initial value problem \eqref{eq:nnls} is locally well-posed in $H^s$ for $s\ge 0$.
In particular, for $s\ge 0$, and any initial data $u_0\in H^s(\mathbb{R})$, there exists a time  $T=T(\|u_0\|_{L^2})$ such that the initial value problem \eqref{eq:nnls} has a unique solution $u(t)$ in the time interval $(-T,T)$ with
\begin{equation*}
u\in C((-T,T):H^1)\cap C^1((-T,T):H^{-1})\cap L^6((-T,T)\times \mathbb{R}).
\end{equation*}
\end{proposition}

Let us remark on the choice of identifying the phase rotation function $\theta(t)$ and the dilation function $\alpha(t)$ in \eqref{eq:explicit}.
For $r>0$, we consider the neighborhood of radius of $r$ around $Q$, modulo phase rotation and dilation
\begin{equation*}
U_{r}=\left\{u\in H^1 \mid d(u,Q)\le r\right\}.
\end{equation*}
In a similar way to \cite[Proposition 1]{martel}, we obtain the following lemma.

\begin{lemma}\label{lem:c^1map}
There exist small $\varepsilon_0>0$ and $\alpha_0>0$ such that there exists a unique $C^1$-map $(\theta,\alpha):U_{\varepsilon_0}\to \mathbb{R}\times (1-\alpha_0,1+\alpha_0)$, such that for $u\in U_{\varepsilon_0}$, we have
\begin{equation*}
\langle i(u-e^{i\theta}Q_{\alpha})\mid e^{i\theta}Q'_{\alpha}\rangle=\langle u-e^{i\theta}Q_{\alpha}\mid e^{i\theta}Q_{\alpha}\rangle= 0.
\end{equation*}
Moreover, there exists $C>0$ such that if $u\in U_{\varepsilon}$ with $\varepsilon\le \varepsilon_0$, then
\begin{equation}\label{eq:theta-diff}
\|u-e^{i\theta}Q_{\alpha}\|_{H^1}\le C\varepsilon,\quad |\alpha-1|\le C\varepsilon.
\end{equation}
\end{lemma}

\begin{proof}
Repeating the argument of \cite[Proposition 1]{martel}, we take functionals
\begin{equation*}
\begin{gathered}
\rho_{\theta,\alpha}^1(u)=\langle i(e^{-i\theta}u_{1/\alpha}-Q)\mid Q'\rangle,\\
\rho_{\theta,\alpha}^2(u)=\langle e^{-i\theta}u_{1/\alpha}-Q\mid Q\rangle,
\end{gathered}
\end{equation*}
where $u_{\alpha}$ is rescaled function as $u_{\alpha}(x)=\alpha u(\alpha x)$.
Since
\begin{equation*}
\begin{gathered}
\rho_{0,1}^1(Q)=\rho_{0,1}^2(Q)=0,\\
\left.\frac{\partial \rho_{\theta,\alpha}^1(u)}{\partial\theta}\right|_{(\theta,\alpha,u)=(0,1,Q)}=\langle Q\mid Q'\rangle=M(Q)>0,\\
\left.\frac{\partial \rho_{\theta,\alpha}^1(u)}{\partial\alpha}\right|_{(\theta,\alpha,u)=(0,1,Q)}=0,\\
\left.\frac{\partial \rho_{\theta,\alpha}^2(u)}{\partial\theta}\right|_{(\theta,\alpha,u)=(0,1,Q)}=0,\\
\left.\frac{\partial \rho_{\theta,\alpha}^2(u)}{\partial\alpha}\right|_{(\theta,\alpha,u)=(0,1,Q)}=M(Q)>0,
\end{gathered}
\end{equation*}
the implicit function theorem implies that there exists $\varepsilon_1>0$, a neighborhood $V_{0,1}$ of $(0,1)$ in $\mathbb{R}^2$ and a unique $C^1$-map $(\theta_1,\alpha_1):\{u\in H^1\mid \|u-Q\|_{H^1}<\varepsilon_1\}\to V_{0,1}~(u\mapsto (\theta_1(u),\alpha_1(u)))$ such that $\rho_{\theta_1,\alpha_1}^1(u)=\rho_{\theta_1,\alpha_1}^2(u)=0$.
Moreover, we see that there exists $C>0$ such that if $\|u-Q\|_{H^1}<\varepsilon_2$ with $\varepsilon_2\le \varepsilon_1$, then  the following estimates hold:
\begin{equation}\label{eq:alpha-diff}
|\theta_1|+|\alpha_1-1|<C\varepsilon_2,\quad\|e^{-i\theta_1}u_{\alpha_1}-Q\|_{H^1}<C\varepsilon_2.
\end{equation}

Once again, the implicit function theorem lets us find $\varepsilon_0\in(0,\varepsilon_1)$ and a unique map $\theta_2:U_{\varepsilon_0}\to\mathbb{R}~(u\mapsto (\theta_2(u)))$ such that for all $u\in U_{\varepsilon_0}$,
$$
\|u-e^{i\theta_2}Q\|_{H^1}=d(u,Q)<\varepsilon_0.
$$
Now put $\alpha(u)=\alpha_1(e^{-i\theta_2}u)$ and $\theta(u)=\theta_1(e^{-i\theta_2(u)}u)+\theta_2(u)$.
The estimate of \eqref{eq:theta-diff} follows from \eqref{eq:alpha-diff}.
Then we obtain the desired result.
\end{proof}

Let us take $u_0\in\mathcal{S}_{\varepsilon}$ for $\varepsilon>0$ so small that Lemma \ref{lem:c^1map} is valid.
Using Proposition \ref{prop:LWP} and Lemma \ref{lem:c^1map}, we have that there exist a constant $C>0$ and the $C^1$-map $(\theta,\alpha)=(\theta(t),\alpha(t))$ as long as the solution satisfies $d(u(t),Q)\le C\varepsilon$.
Moreover,
\begin{equation}\label{eq:constrain}
\langle iv(t)\mid Q'_{\alpha(t)}\rangle=\langle v(t)\mid Q_{\alpha(t)}\rangle=0,
\end{equation}
and
\begin{equation*}
\|v(t)\|_{H^1}\le C\varepsilon,\quad |\alpha(t)-1|\le C\varepsilon.
\end{equation*}

Substitution \eqref{eq:explicit} into \eqref{eq:nnls} yields the following equation:
\begin{equation}\label{eq:rnnls}
i\partial_t v+\mathcal{L}v=(1+\dot{\theta})v+(\alpha^2+\dot{\theta})Q_{\alpha}-i\dot{\alpha}Q_{\alpha}'+(Q_{\alpha(t)}^2-Q^2)(2v+v^{\star})+\mathcal{N}_{\alpha}(v,v^{\star}),
\end{equation}
where $\mathcal{L}$ is the linearized operator of \eqref{eq:nnls} at ground states
\begin{equation*}
\mathcal{L}v(t,x)  =   -\partial_x^2v(t,x)+v(t,x)-2Q(x)^2v(t,x)-Q(x)^2v^{\star}(t,x)
\end{equation*}
and nonlinear terms are as follows
\begin{equation*}
\mathcal{N}_{\alpha}(v,v^{\star})(t,x)  =N_{\alpha,1}(v)(t,x)+N_{\alpha,2}(v,v^{\star})(t,x)+N_{3}(v,v^{\star})(t,x),
\end{equation*}
\begin{equation*}
\begin{gathered}
N_{\alpha,1}(v_1)(t,x)=  Q_{\alpha(t)}(x)v_1(t,x)^2,\\
N_{\alpha,2}(v_1,v_2)(t,x)=  2Q_{\alpha(t)}(x)v_1(t,x)v_2(t,x),\\
N_{3}(v_1,v_2)(t,x)=  v_1(t,x)^2v_2(t,x).
\end{gathered}
\end{equation*}

Decomposing the function $v$ into even and odd parts as 
\begin{equation*}
v_e=v_e(t,x)=\frac{v(t,x)+v(t,-x)}{2},\quad
v_o=v_o(t,x)=\frac{v(t,x)-v(t,-x)}{2},
\end{equation*}
we will rewrite the equation \eqref{eq:rnnls} to the formula of linearized equations associated to \eqref{eq:nls}.
By \eqref{eq:rnnls}, $v_e$ and $v_o$ satisfy
\begin{equation}\label{eq:fg0}
i\dot{v}_{e}+\mathcal{L}_{-}v_e=(1+\dot{\theta})v_e+(\alpha^2+\dot{\theta})Q_{\alpha}-i\dot{\alpha}Q_{\alpha}'+(Q_{\alpha}^2-Q^2)(2v_e+\overline{v_e})+\mathcal{N}_{\alpha,e}(v_e,v_o),
\end{equation}
\begin{equation}\label{eq:fg1}
i\dot{v}_o+\mathcal{L}_{+}v_o=(1+\dot{\theta})v_o+(Q_{\alpha}^2-Q^2)(2v_o-\overline{v_o})+\mathcal{N}_{\alpha,o}(v_e,v_o),
\end{equation}
where $(\theta,\alpha)=(\theta(t),\alpha(t))$ and
\begin{equation}\label{eq:N_e}
\begin{split}
\mathcal{N}_{\alpha,e}(v_e,v_o)= & \frac12\left(\mathcal{N}_{\alpha}(v_e+v_o,\overline{v_e-v_o})+ \mathcal{N}_{\alpha}(v_e-v_o,\overline{v_e+v_o})\right)\\
= & Q_{\alpha}(v_e^2+v_o^2)+2Q_{\alpha}(|v_e|^2-|v_o|^2)+O(|v_e|^3+|v_o|^3)
\end{split}
\end{equation}
and
\begin{equation}\label{eq:N_o}
\begin{split}
\mathcal{N}_{\alpha,o}(v_e,v_o)= & \frac12\left(\mathcal{N}_{\alpha}(v_e+v_o,\overline{v_e-v_o})- \mathcal{N}_{\alpha}(v_e-v_o,\overline{v_e+v_o})\right)\\
= & 2Q_{\alpha}v_ev_o-4Q_{\alpha}\Im(v_e\overline{v_o})+O(|v_e|^3+|v_o|^3).
\end{split}
\end{equation}

Consider $V_e=(v_e,\overline{v_e})^t$ and $V_o=(v_o,\overline{v_o})^t$, and insert these into \eqref{eq:fg0}.
Then the equations \eqref{eq:fg0} turn to the system of equations as follows
\begin{equation}\label{eq:Mfg}
\begin{split}
i\partial_t V_e
+\mathcal{H}_{e} V_e
=& 
  (1+\dot{\theta})\sigma_3V_e+
  \left((\dot{\theta}+\alpha^2)Q_{\alpha}\sigma_3+i\dot{\alpha}Q_{\alpha}'\right)
(1,1)^t\\
  & +
  (Q_{\alpha}^2-Q^2)\sigma_3
  \begin{pmatrix}
   2 & 1 \\
    1 & 2
  \end{pmatrix}
V_e  
  +
\sigma_3
 \left ( \mathcal{N}_{\alpha,e}(v_e,v_o),   \overline{ \mathcal{N}_{\alpha,e}(v_e,v_o)}\right)^t
\end{split}
  \end{equation}
and
\begin{equation}\label{eq:Mfg1}
\begin{split}
i\partial_t V_o
+\mathcal{H}_{o} V_o
= &
 (1+\dot{\theta})\sigma_3 V_o  
 +
  (Q_{\alpha}^2-Q^2)\sigma_3
  \begin{pmatrix}
   2 & -1 \\
    -1 & 2
  \end{pmatrix}
V_o  \\
 & +
 \sigma_3
\left ( \mathcal{N}_{\alpha,o}(v_e,v_o) ,
     \overline{\mathcal{N}_{\alpha,o}(v_e,v_o)}\right)^t
\end{split}
\end{equation}
where
\begin{equation*}
\mathcal{H}_{e}=
 (-\partial_x^2+1)\sigma_3+ Q^2\begin{pmatrix}
     -2 & -1 \\
     1 & 2
  \end{pmatrix}
\end{equation*}
and $\mathcal{H}_{o}=\mathcal{H}^t_{e}$.

\subsection{Linearized operator}

In this subsection, we consider the linearized eigenvalue problem derived from the system \eqref{eq:Mfg}-\eqref{eq:Mfg1}.

The matrix Schr\"odinger operator $\mathcal{H}_{e}$ is non self-adjoint in the product space $L^2\times L^2$.
Let $\sigma_d(\mathcal{H}_{e})$ and $\sigma_{\mathrm{ess}}(\mathcal{H}_{e})$ be the discrete, essential spectrums to $\mathcal{H}_{e}$, respectively.
Then we have that the spectrum of $\sigma(\mathcal{H}_{e})$ is decomposed into disjoint pairs $\sigma_d(\mathcal{H}_{e})$ and $\sigma_{\mathrm{ess}}(\mathcal{H}_{e})$ as $\sigma(\mathcal{H}_{e})=\sigma_d(\mathcal{H}_{e})\cup\sigma_{\mathrm{ess}}(\mathcal{H}_{e})$.
The same holds for the operator $\mathcal{H}_{o}$. 

As usual, we use
$$
\mathcal{P}=\frac{1}{\sqrt{2}}
 \begin{pmatrix}
     1 & i \\
     1 & -i
  \end{pmatrix}
$$
to be an unitary matrix such that $\mathcal{P}\overline{\mathcal{P}^t}=  \overline{\mathcal{P}^t}\mathcal{P}=I$.
It follows from the straightforward calculation that
\begin{equation}\label{eq:Pt}
\begin{gathered}
\mathcal{P}^{-1}{\mathcal H}_{e} \mathcal{P}
=
i \begin{pmatrix}
     0 & L_{- }\\
     -L_{+} & 0
       \end{pmatrix},\\       
\mathcal{P}^{-1}{\mathcal H}_{o}\mathcal{P}
=
 -i\begin{pmatrix}
     0 & -L_{+} \\
     L_{-} & 0
       \end{pmatrix}.
       \end{gathered}
\end{equation}
We take advantage of the scaled unitary matrix $\mathcal{P}$.
By the expedient way done in \cite{weinstein1} also in \cite{nakanishi}, we work with real and imaginary parts of $v_e$ and $v_o$, respectively.
From the argument in \cite{weinstein1}, $L_{-}$ is nonnegative and $\ker L_{-}=\mathrm{span}\{Q\}$.
Moreover, $L_{+}(\partial_x Q)=0$, in which the function $\partial_x Q$ is odd.

Let us refer the results in \cite[Lemma 2.3]{chang} and \cite[Theorem B.3]{weinstein1}.
As results by \eqref{eq:Pt}, the following proposition is provided without proof, where $Q'$ means $Q'=(1+x\partial_x)Q$ defined as before.

\begin{proposition}[\cite{chang,weinstein1}]\label{prop:eighenvalue}
The essential spectrums of $\mathcal{H}_{e}$ and $\mathcal{H}_{o}$ are $(-\infty,-1]\cup[1,\infty)$ and there are no embedded eigenvalues or resonances in $(-\infty,-1)\cup(1,\infty)$, the thresholds $\pm 1$ are of the resonances, the root space at $0$ is of the dimension two, respectively.
In explicit form in \cite{weinstein1}, the root space of $\mathcal{P}^{-1}\mathcal{H}_{e}\mathcal{P}$ for even functions is
\begin{equation}\label{eq:rootH}
\mathrm{span}\left\{ \begin{pmatrix}
     0 \\
     Q 
       \end{pmatrix},~
 \begin{pmatrix}
     Q' \\
     0
       \end{pmatrix}
       \right\}.
\end{equation}
Moreover the corresponding root space of $\mathcal{P}^{-1}\mathcal{H}_{o}\mathcal{P}$ for odd functions is
\begin{equation}\label{eq:rootHt}
\mathrm{span}\left\{ 
\begin{pmatrix}
     0  \\
     \partial_xQ
       \end{pmatrix},~
 \begin{pmatrix}
     xQ \\
     0
       \end{pmatrix}
       \right\}.
\end{equation}
\end{proposition}

\begin{remark}
The root space of $\mathcal{P}^{-1}\mathcal{H}_{e}\mathcal{P}$ for odd functions is
\begin{equation*}
\mathrm{span}\left\{ \begin{pmatrix}
     \partial_xQ \\
      0
       \end{pmatrix},~
 \begin{pmatrix}
      0\\
     xQ
       \end{pmatrix}
       \right\}.
\end{equation*}
\end{remark}

\begin{remark}
It is well-known that the essential spectrum of $\mathcal{H}_e$ equals to $(-\infty,-1]\cup[1,\infty)$, and there are no eigenvalues in the essential spectrum (see Krieger and Schlag \cite{krieger}).
The presence of a resonance at the endpoint of the essential spectrum was studied in detail by Chang, Gustafson, Nakanishi and Tsai \cite{chang}.
We easily see that the relation \eqref{eq:Pt} implies the same thing as above holds for $\mathcal{H}_o$.
\end{remark}

\begin{remark}
From a technical perspective, we decompose the modulation perturbation into even and odd parts $(V_e,V_o)$, since the linearized operators for $(V_e,V_o)$ have better structures compared to the original operator $\mathcal{L}$. 
Indeed, these new operators are transformed into well-known $2\times 2$ operators \eqref{eq:Pt} whose spectral properties have extensively clarified.
This trick enable us to carry out the framework of \cite{krieger}.
\end{remark}

Keeping the above observation, we have the orthogonal projection operators on the range of continuous spectrum and of point spectrum.

\begin{definition}\label{def:orthogonal}
Let $P_{e,d}$ denotes the projector onto the root spaces \eqref{eq:rootH} of $\mathcal{H}_{e}$.
Define $P_{e,s}=I-P_{e,d}$.
In a similar, we define $P_{o,d},~P_{o,s}$ for $\mathcal{H}_{o}$ by \eqref{eq:rootHt}, respectively.
\end{definition}

The solution spaces for $v_e$ and $v_o$ of \eqref{eq:Mfg}-\eqref{eq:Mfg1} have the symplectic decompositions of $L^2(\mathbb{R})$ in the sense by Definition \ref{def:orthogonal}.
Namely, for even function $v_e$ and odd function $v_o$, we can write $v_e$ and $v_o$ as the following symplectic decompositions of $L^2(\mathbb{R})$ with respect to $\omega_{NLS}$ in \eqref{eq:Omega}, in the same sense to $\omega$
\begin{equation}\label{eq:fg}
\begin{gathered}
v_e(t,x)=a_e(t)i Q(x)+b_e(t)Q'(x)+\eta_{e}(t,x),\\
v_o(t,x)=a_o(t)i \partial_xQ(x)+b_o(t)xQ(x)+\eta_{o}(t,x),
\end{gathered}
\end{equation}
where
$$
a_e(t)=-\frac{\langle iv_e(t)\mid Q'\rangle}{\langle Q\mid Q'\rangle},\quad b_e(t)=\frac{\langle v_e(t)\mid Q\rangle}{\langle Q'\mid Q\rangle},
$$
$$
a_o(t)=-\frac{\langle iv_o(t)\mid xQ\rangle}{\langle \partial_xQ\mid xQ\rangle},\quad b_o(t)=\frac{\langle v_o(t)\mid \partial_xQ\rangle}{\langle xQ\mid \partial_x Q\rangle}
$$
are real-valued functions.
Here we decompose the dynamics of solutions $V_e$ (respectively $V_o$) for the system \eqref{eq:Mfg} (respectively \eqref{eq:Mfg1}) in three components by using \eqref{eq:rootH}  (respectively \eqref{eq:rootHt}).
Let $p_{e,d}v_e,~p_{e,s}v_e,~p_{o,d}v_o,~p_{o,s}v_o$ denote the first components of $P_{e,d}V_e,~ P_{e,s}V_e,~P_{o,d}V_o,~P_{o,s}V_o$, respectively.
From Definition \ref{def:orthogonal}, we have
\begin{equation*}
\begin{gathered}
p_{e,d}v_e(t,x)=a_e(t)i Q(x)+b_e(t)Q'(x),\\
p_{e,s}v_e(t,x)=\eta_{e}(t,x),\\
p_{o,d}v_o(t,x)=a_o(t)i \partial_xQ(x)+b_o(t)xQ(x),\\
p_{o,s}v_o(t,x)=\eta_{o}(t,x).
\end{gathered}
\end{equation*}

We would like to summarize the following useful computations
\begin{equation}\label{eq:relation}
\begin{gathered}
\langle Q\mid Q'\rangle=M(Q),\quad
\langle \partial_xQ\mid xQ\rangle =-M(Q).
\end{gathered}
\end{equation}
To obtain the coefficients in \eqref{eq:fg}, the relations \eqref{eq:relation} yield
\begin{equation}\label{eq:p-f}
\begin{gathered}
a_e(t)=-\frac{\langle iv_e(t)\mid Q'\rangle}{M(Q)},\quad
b_e(t)=\frac{\langle v_e\mid Q\rangle}{M(Q)},\\
\langle \eta_e(t)\mid Q\rangle =\langle i\eta_e(t)\mid Q'\rangle =0
\end{gathered}
\end{equation}
and
\begin{equation}\label{eq:p-g}
\begin{gathered}
a_o(t)=\frac{\langle iv_o(t)\mid xQ\rangle}{M(Q)},\quad
b_o(t)=-\frac{\langle v_o\mid \partial_xQ\rangle}{M(Q)},\\
\langle \eta_o(t)\mid \partial_xQ\rangle =\langle i\eta_o(t)\mid xQ\rangle =0.
\end{gathered}
\end{equation}

\section{Parameter choices in evolutionary computation}
 
For estimating the prospects of $a_e(t),~b_e(t),~\eta_e(t),~a_o(t),~b_o(t),~\eta_o(t),~\theta(t),~\alpha(t)$, we will use a standard bootstrap argument, which one sees in the next section.
Specifically, we assume the following estimates hold on some time interval in this section:
$$
|a_e(t)|+|b_e(t)|+\|\eta_e(t)\|_{H^1}+|a_o(t)|+|b_o(t)|+\|\eta_o(t)\|_{H^1}\ll 1
$$
and
$$
|\alpha(t)^2-1|\ll 1,~|1+\dot{\theta}(t)|\ll 1.
$$

\subsection{Computations on  $a_e(t),~b_e(t),~a_o(t),~b_o(t),~\dot{\theta}(t),~\dot{\alpha}(t)$}

Firstly, we consider terms $a_e(t)$ and $b_e(t)$.
Since $v_e$ is the even function of $x$, we easily see that by \eqref{eq:constrain}
\begin{equation}\label{eq:a_e}
a_e(t)=\frac{\langle iv_e(t)\mid Q_{\alpha(t)}'-Q'\rangle}{M(Q)}.
\end{equation}
By differentiation along with the equation \eqref{eq:fg0}, we have
\begin{equation}\label{eq:a_ed}
\begin{split}
&M(Q)\dot{a}_e(t)\\
& =\langle \mathcal{L}_-v_e(t)\mid Q'\rangle-(1+\dot{\theta}(t))\langle v_e(t)\mid Q'\rangle-(\alpha(t)^2+\dot{\theta}(t))\langle Q_{\alpha(t)}\mid Q'\rangle\\
&\quad  -\langle (Q_{\alpha(t)}^2-Q^2)(2v_e(t)+\overline{v}_e(t))\mid Q'\rangle-\langle \mathcal{N}_{\alpha(t),e}(v_e(t),v_o(t)) \mid Q'\rangle\\
& =\langle L_+v_e(t)\mid Q'\rangle-(1+\dot{\theta}(t))\langle v_e(t)\mid Q'\rangle-(\alpha(t)^2+\dot{\theta}(t))\langle Q_{\alpha(t)}\mid Q'\rangle\\
&\quad  -3\langle (Q_{\alpha(t)}^2-Q^2)v_e(t)\mid Q'\rangle-\langle \mathcal{N}_{\alpha(t),e}(v_e(t),v_o(t)) \mid Q'\rangle\\
& = -2\langle v_e(t)\mid Q_{\alpha(t)}-Q\rangle-(1+\dot{\theta}(t))\langle v_e(t)\mid Q'\rangle-(\alpha(t)^2+\dot{\theta}(t))\langle Q_{\alpha(t)}\mid Q'\rangle\\
&\quad  -3\langle (Q_{\alpha(t)}^2-Q^2)v_e(t)\mid Q'\rangle-\langle \mathcal{N}_{\alpha(t),e}(v_e(t),v_o(t)) \mid Q'\rangle\\
& =O\left(|\alpha(t)-1|\|v_e(t)\|_{H^1}+|\alpha(t)^2+\dot{\theta}(t)|+|1+\dot{\theta}(t)|\|v_e(t)\|_{H_1}+\|v_e\|_{H^1}^2+\|v_o\|_{H^1}^2\right),
\end{split}
\end{equation}
where we employ $L_+Q'=-2Q$ and $\langle v_e(t)\mid Q_{\alpha(t)}\rangle =0$.

On the other hand, again by \eqref{eq:constrain}
\begin{equation}\label{eq:b_e}
b_e(t)=-\frac{\langle v_e(t)\mid Q_{\alpha(t)}-Q\rangle}{M(Q)}.
\end{equation}
Repeated computation gives
\begin{equation}\label{eq:b_ed}
\begin{split}
&M(Q)\dot{b}_e(t)\\
& =\langle i\mathcal{L}_-v_e(t)\mid Q\rangle-(1+\dot{\theta}(t))\langle iv_e(t)\mid Q\rangle\\
&\quad  -\dot{\alpha}(t)\langle Q_{\alpha(t)}'\mid Q\rangle -\langle i(Q_{\alpha(t)}^2-Q^2)v_e(t)\mid Q\rangle-\langle i\mathcal{N}_{\alpha(t),e}(v_e(t),v_o(t)) \mid Q\rangle\\
& =\langle iL_-v_e(t)\mid Q\rangle-(1+\dot{\theta}(t))\langle iv_e(t)\mid Q\rangle\\
&\quad  -\dot{\alpha}(t)\langle Q_{\alpha(t)}'\mid Q\rangle -\langle i(Q_{\alpha(t)}^2-Q^2)v_e(t)\mid Q\rangle-\langle i\mathcal{N}_{\alpha(t),e}(v_e(t),v_o(t)) \mid Q\rangle\\
& = -(1+\dot{\theta}(t))\langle iv_e(t)\mid Q\rangle\\
&\quad  -\dot{\alpha}(t)\langle Q_{\alpha(t)}'\mid Q\rangle -\langle i(Q_{\alpha(t)}^2-Q^2)v_e(t)\mid Q\rangle-\langle i\mathcal{N}_{\alpha(t),e}(v_e(t),v_o(t)) \mid Q\rangle\\
& =O\left(|\dot{\alpha}(t)|+|\alpha(t)-1|\|v_e(t)\|_{H^1}+|1+\dot{\theta}(t)|\|v_e(t)\|_{H^1}+\|v_e\|_{H^1}^2+\|v_o\|_{H^1}^2\right),
\end{split}
\end{equation}
where $L_-Q=0$ was used in the above computations.

Secondly, we treat terms $a_o(t)$ and $b_o(t)$.
By \eqref{eq:fg1} and
\begin{equation*}
L_-(xQ)=xL_-Q-2\partial_xQ=-2\partial_xQ,
\end{equation*}
it follows that
\begin{equation}\label{eq:a_o}
\begin{split}
&\dot{a}_o(t)M(Q)\\
& = \langle -\mathcal{L}_+v_o(t) \mid xQ\rangle+(1+\dot{\theta}(t))\langle v_0(t) \mid xQ\rangle\\
& \quad +\langle (Q_{\alpha(t)}^2-Q^2)(2v_o(t)-\overline{v_o(t)}) \mid xQ\rangle+\langle \mathcal{N}_{\alpha(t),o}(v_e(t),v_o(t)) \mid xQ\rangle\\
& = -\langle L_-v_o(t)\mid xQ\rangle +(1+\dot{\theta}(t))\langle v_o(t)\mid xQ\rangle\\
& \quad+\langle (Q_{\alpha(t)}^2-Q^2)v_o(t)\mid xQ\rangle +\langle\mathcal{N}_{\alpha(t),o}(v_e(t),v_o(t)) \mid xQ\rangle\\
& = -2b_o(t)M(Q)+(1+\dot{\theta}(t))\langle v_o(t)\mid xQ\rangle + \langle (Q_{\alpha(t)}^2-Q^2)v_o(t)\mid xQ\rangle\\
& \quad + \langle\mathcal{N}_{\alpha(t),o}(v_e(t),v_o(t)) \mid xQ\rangle\\
& =-2b_o(t)M(Q)+O\left(|\alpha(t)-1|\|v_o(t)\|_{H^1}+|1+\dot{\theta}(t)|\|v_o(t)\|_{H^1}+\|v_e\|_{H^1}^2+\|v_o\|_{H^1}^2\right),
\end{split}
\end{equation}
Again using  \eqref{eq:fg1}, we have
\begin{equation}\label{eq:b_o}
\begin{split}
& \dot{b}_o(t) M(Q)\\
& =-\langle i\mathcal{L}_+v_o(t)\mid \partial_xQ\rangle+(1+\dot{\theta}(t))\langle iv_o(t)\mid\partial_xQ\rangle\\
& \quad +3\langle i(Q_{\alpha(t)}^2-Q^2)v_o(t)\mid \partial_xQ\rangle +\langle i\mathcal{N}_{\alpha(t),o}(v_e(t),v_o(t))\mid \partial_xQ\rangle\\
& = -\langle iL_+v_o(t)\mid \partial_xQ\rangle+(1+\dot{\theta}(t))\langle iv_o(t)\mid\partial_xQ\rangle\\
& \quad +3\langle i(Q_{\alpha(t)}^2-Q^2)v_o(t)\mid \partial_xQ\rangle +\langle i\mathcal{N}_{\alpha(t),o}(v_e(t),v_o(t))\mid \partial_xQ\rangle\\
& = (1+\dot{\theta}(t))\langle iv_o(t)\mid\partial_xQ\rangle+3\langle i(Q_{\alpha(t)}^2-Q^2)v_o(t)\mid \partial_xQ\rangle\\
& \quad +\langle i\mathcal{N}_{\alpha(t),o}(v_e(t),v_o(t))\mid \partial_xQ\rangle\\
& =O\left(|\alpha(t)-1|\|v_o(t)\|_{H^1}+|1+\dot{\theta}(t)|\|v_o(t)\|_{H^1}+\|v_e\|_{H^1}^2+\|v_o\|_{H^1}^2\right),
\end{split}
\end{equation}

Thirdly, we deal with the term $\alpha(t)^2+\dot{\theta}(t)$. 
Taking the product in \eqref{eq:fg0} with $Q_{\alpha(t)}'$, we have
\begin{equation*}
\begin{split}
& \langle i\dot{v_e}(t)\mid Q_{\alpha(t)}'\rangle+\langle \mathcal{L}_-v_e(t)\mid Q_{\alpha(t)}'\rangle\\
& =(1+\dot{\theta}(t))\langle v_e(t) \mid Q_{\alpha(t)}'\rangle+(\alpha(t)^2+\dot{\theta}(t))\langle Q_{\alpha(t)}\mid Q_{\alpha(t)}'\rangle\\
& \quad +3\langle (Q_{\alpha(t)}^2-Q^2)v_e(t)\mid Q_{\alpha(t)}'\rangle+\langle \mathcal{N}_{\alpha(t),e}(v_e(t),v_o(t))\mid Q_{\alpha(t)}'\rangle.
\end{split}
\end{equation*}
On the left-hand side, we compute by \eqref{eq:constrain}
\begin{equation*}
\langle i\dot{v_e}(t)\mid Q_{\alpha(t)}'\rangle  =-\dot{\alpha}(t)\langle iv_{e}(t)\mid Q_{\alpha(t)}''\rangle
=O\left(|\dot{\alpha}(t)|\|v_e(t)\|_{H^1}\right),
\end{equation*}
\begin{equation*}
\begin{split}
\langle \mathcal{L}_-v_e(t)\mid Q_{\alpha(t)}'\rangle & =\langle L_+v_e(t)\mid Q_{\alpha(t)}'-Q'\rangle+\langle L_+v_e(t)\mid Q'\rangle\\
& =\langle v_e(t)\mid L_+(Q_{\alpha(t)}'-Q')\rangle-2\langle v_e(t)\mid Q\rangle \\
& =\langle v_e(t)\mid L_+(Q_{\alpha(t)}'-Q')\rangle -2M(Q)b_e(t)\\
& = -2M(Q)b_e(t)+O\left(|\alpha(t)-1|\|v_e(t)\|_{H^1}\right)
\end{split}
\end{equation*}
and
\begin{equation*}
\begin{split}
\langle Q_{\alpha(t)} \mid Q_{\alpha(t)}'\rangle = M(Q)+O\left(|\alpha(t)-1|\right).
\end{split}
\end{equation*}
Then we obtain
\begin{equation}\label{eq:theta}
\begin{split}
& \left(\alpha(t)^2+\dot{\theta}(t)\right)\left(M(Q)+O\left(|\alpha(t)-1|+\|v_e(t)\|_{H^1}\right)\right)\\
& = 2M(Q)b_e(t)+O\left(|\dot{\alpha}(t)|\|v_e(t)\|_{H^1}+|\alpha(t)-1|\|v_e(t)\|_{H^1}+\|v_e(t)\|_{H^1}^2+\|v_o(t)\|_{H^1}^2\right).
\end{split}
\end{equation}

So finally, let us consider the term $\dot{\alpha}(t)$.
By \eqref{eq:constrain}
\begin{equation*}
\langle \dot{v}_e(t)\mid Q_{\alpha(t)}\rangle=-\dot{\alpha}(t)\langle v_e(t)\mid  Q_{\alpha(t)}'\rangle.
\end{equation*}
Then \eqref{eq:fg0} and $L_-Q=0$ imply
\begin{equation*}
\begin{split}
&-\dot{\alpha}(t)\langle v_e(t)\mid  Q_{\alpha(t)}'\rangle\\
& =  \langle iL_-v_e(t)\mid Q_{\alpha(t)}\rangle-(1+\dot{\theta}(t))\langle iv_e(t)\mid Q_{\alpha(t)}\rangle -\dot{\alpha}(t)\langle Q_{\alpha(t)}'\mid Q_{\alpha(t)}\rangle\\
& \quad -\langle i(Q_{\alpha(t)}^2-Q^2)v_e\mid Q_{\alpha(t)}\rangle-\langle iN_{\alpha(t),e}(v_e(t),v_o(t))\mid Q_{\alpha(t)}\rangle\\
& =  \langle iv_e(t)\mid L_-(Q_{\alpha(t)}-Q)\rangle-(1+\dot{\theta}(t))\langle iv_e(t)\mid Q_{\alpha(t)}\rangle \\
& \quad -\dot{\alpha}(t)\left(M(Q)+\langle Q_{\alpha(t)}'-Q'\mid Q_{\alpha(t)}\rangle+\langle Q'\mid Q_{\alpha(t)}-Q\rangle\right)\\
&\quad +\langle i(Q_{\alpha(t)}^2-Q^2)v_e\mid Q_{\alpha(t)}\rangle+ \langle iN_{\alpha(t),e}(v_e(t),v_o(t))\mid Q_{\alpha(t)}\rangle\\
& =-\dot{\alpha}(t)\left(M(Q)+O\left(|\alpha(t)-1|\right)\right)\\
&\quad +O\left(|\alpha(t)-1|\|v_e(t)\|_{H^1}+|1+\dot{\theta}(t)|\|v_e(t)\|_{H^1}+\|v_e(t)\|_{H^1}^2+\|v_o(t)\|_{H^1}^2\right),
\end{split}
\end{equation*}
which leads to us
\begin{equation}\label{eq:alpha}
\begin{split}
& \dot{\alpha}(t)\left(M(Q)+O\left(|\alpha(t)-1|+\|v_e(t)\|_{H^1}\right)\right)\\
& =O\left(|\alpha(t)-1|\|v_e(t)\|_{H^1}+|1+\dot{\theta}(t)|\|v_e(t)\|_{H^1}+\|v_e(t)\|_{H^1}^2+\|v_o(t)\|_{H^1}^2\right).
\end{split}
\end{equation}

\subsection{Computations on $\eta_e(t),~\eta_o(t)$}

In this subsection, we analyze the estimates on terms $\eta_e(t)$ and $\eta_o(t)$ by the energy method.

Note that 
\begin{equation*}
\frac{\partial}{\partial t}P_{e,d}V_e(t)=
 \begin{pmatrix}
\dot{a}_e(t)iQ+\dot{b}_e(t)Q'\\
-\dot{a}_e(t)iQ+\dot{b}_e(t)Q'
 \end{pmatrix}.
 \end{equation*}
On the other hand, by $L_-Q=0$ and $L_+Q'=-2Q$, we obtain
\begin{equation*}
\begin{split}
\langle i\mathcal{H}_eP_{e,d}V_e(t,x)\mid P_{e,s}V_e(t,x)\rangle & =\langle -2ib_e(t)Q(x)(1,-1)^t,(\eta_e(t),\overline{\eta}_e(t))^t\rangle\\
& =2b_e(t)\Im((Q,-Q)^t,(\eta_e(t),\overline{\eta}_e(t))^t)\\
& = -4b_e(t)\Im( \eta_e(t),Q)
\end{split}
\end{equation*}
and
\begin{equation*}
\begin{split}
\langle (\partial_t-i\mathcal{H}_e)P_{e,s}V_e(t)\mid P_{e,s}V_e(t)\rangle &  =\frac12\frac{d}{dt}\|P_{e,s}V_e(t)\|_{L^2}^2-2\langle
i\eta_e(t)^2\mid Q^2\rangle\\
& = \frac{d}{dt}\|\eta_e(t)\|_{L^2}^2-2\langle i\eta_e(t)^2\mid Q^2\rangle.
\end{split}
\end{equation*}
Taking the product \eqref{eq:Mfg} with $P_{e,s}V_e(t)$, we have
\begin{equation}\label{eq:eta_e}
\begin{split}
&\frac{d}{dt}\|\eta_e(t)\|_{L^2}^2\\
& \le 4\|\eta_e(t)\|_{L^2}^2+c(|\dot{a}_e(t)|+|\dot{b}_e(t)|+|b_e(t)|)\|\eta_e(t)\|_{L^2}+c|1+\dot{\theta}(t)|\|v_e(t)\|_{L^2}\|\eta_e(t)\|_{L^2}\\
& \qquad +|\alpha(t)^2+\dot{\theta}(t)|\|Q_{\alpha(t)}\|_{L^2}\|\eta_e(t)\|_{L^2}+c|\dot{\alpha}(t)|\|Q_{\alpha(t)}'\|_{L^2}\|\eta_e(t)\|_{L^2}\\
& \qquad + c\|Q_{\alpha(t)}^2-Q^2\|_{L^{\infty}}\|v_e(t)\|_{L^2}\|\eta_e(t)\|_{L^2}+O(\|v_e(t)\|_{L^2}^2+\|v_o(t)\|_{L^2}^2)\|\eta_e(t)\|_{L^2}\\
& \le  4\|\eta_e(t)\|_{L^2}^2+c\|\eta_e(t)\|_{L^2}\left(|\dot{a}_e(t)|+|\dot{b}_e(t)|+|b_e(t)|+|1+\dot{\theta}(t)|\|v_e(t)\|_{L^2}\right.\\
& \qquad \left.+|\alpha(t)^2+\dot{\theta}(t)|+|\dot{\alpha}(t)|+O(\|v_e(t)\|_{L^2}^2+\|v_o(t)\|_{L^2}^2)\right).
\end{split}
\end{equation}

Similarly, by $L_+\partial_xQ=0$ and $L_-(xQ)=-2\partial_xQ$, we obtain
\begin{equation*}
\frac{\partial}{\partial t}P_{o,d}V_o(t)=
 \begin{pmatrix}
 \dot{a}_o(t)\partial_xQ+\dot{b}_o(t)xQ\\
 -\dot{a}_o(t)i\partial_xQ+\dot{b}_o(t)xQ
  \end{pmatrix},
\end{equation*}
\begin{equation*}
\begin{split}
\langle i\mathcal{H}_oP_{o,d}V_o(t,x)\mid P_{o,s}V_o(t,x)\rangle & =\langle -2ib_o(t)\partial_xQ(x)(1,-1)^t,(\eta_o(t),\overline{\eta}_o(t))^t\rangle\\
& =2b_o(t)\Im((\partial_xQ,-\partial_xQ)^t,(\eta_o(t),\overline{\eta}_o(t))^t)\\
& = -4b_o(t)\Im( \eta_o(t),\partial_xQ)
\end{split}
\end{equation*}
and
\begin{equation*}
\begin{split}
\langle (\partial_t-i\mathcal{H}_o)P_{o,s}V_o(t)\mid P_{o,s}V_o(t)\rangle &  =\frac12\frac{d}{dt}\|P_{o,s}V_o(t)\|_{L^2}^2-2\langle
i\eta_o(t)^2\mid Q^2\rangle\\
& = \frac{d}{dt}\|\eta_o(t)\|_{L^2}^2-2\langle i\eta_o(t)^2\mid Q^2\rangle.
\end{split}
\end{equation*}
Thus by \eqref{eq:Mfg1} we obtain
\begin{equation}\label{eq:eta_o}
\begin{split}
&\frac{d}{dt}\|\eta_o(t)\|_{L^2}^2\\
&\le  4\|\eta_o(t)\|_{L^2}^2+c\|\eta_o(t)\|_{L^2}\left(|\dot{a}_o(t)|+|\dot{b}_o(t)|+|b_o(t)|+|1+\dot{\theta}(t)|\|\eta_o(t)\|_{L^2}\right.\\
& \qquad \left.+ \|Q_{\alpha(t)}^2-Q^2\|_{L^{\infty}}\|v_o\|_{L^2}+O(\|v_e(t)\|_{L^2}^2+\|v_o(t)\|_{L^2}^2)\right)\\
& \le 4\|\eta_o(t)\|_{L^2}^2+c\|\eta_o(t)\|_{L^2}\left(|\dot{a}_o(t)|+|\dot{b}_o(t)|+|b_o(t)|+|1+\dot{\theta}(t)|\|\eta_o(t)\|_{L^2}\right.\\
& \left. \qquad +|\alpha(t)-1|\|v_o\|_{L^2}+O(\|v_e(t)\|_{L^2}^2+\|v_o(t)\|_{L^2}^2)\right).
\end{split}
\end{equation}

\section{Proof of Theorem \ref{thm:main}}\label{sec:proof}

In this section, we give the a priori estimate of solution $v(t)$ to \eqref{eq:rnnls} that accounts for the evolution of $a_e(t),~b_e(t),~\eta_e(t),~a_o(t),~b_o(t),~\eta_o(t)$ in \eqref{eq:p-g}.
The following proposition holds.

\begin{proposition}\label{prop:apriori}
Let $\varepsilon>0$ be a small number, and $(\theta(t),\alpha(t))$ be obtained in Lemma $\ref{lem:c^1map}$ such that \eqref{eq:constrain}.
Let $a_e(t),~b_e(t),~\eta_e(t),~a_o(t),~b_o(t),~\eta_o(t)$ be defined in \eqref{eq:fg}.
Suppose that
\begin{equation}\label{eq:data}
\begin{gathered}
|a_e(0)|+|b_e(0)|+|a_o(0)|\lesssim \varepsilon,\\
|b_o(0)|+\|\eta_e(0)\|_{H^1}+\|\eta_o(0)\|_{H^1}\lesssim \varepsilon^{2}.
\end{gathered}
\end{equation}
Then there exists $T\ll \log(1/\varepsilon)$ for which initial value problems \eqref{eq:Mfg} admit unique solutions $v_e(t),~v_o(t)\in C([-T,T]:H^1)$ of the form \eqref{eq:fg} with initial data
\begin{equation*}
\begin{gathered}
v_e(0,x)=a_e(0)i Q(x)+b_e(0)Q'(x)+\eta_e(0,x),\\
v_o(0,x)=a_o(0)i \partial_xQ(x)+b_o(0)xQ(x)+\eta_o(0,x),
\end{gathered}
\end{equation*}
respectively.
Moreover the following bounds hold for above solutions
\begin{equation*}
\begin{split}
\sup_{|t|\le T}& \left(|a_e(t)|+|b_e(t)|+|a_o(t)|+|b_o(t)|\right.\\
& \left.+\|\eta_e(t)\|_{H^1}+\|\eta_o(t)\|_{H^1}+|\alpha(t)-1|\right)\lesssim \varepsilon
\end{split}
\end{equation*}
and
\begin{equation*}
\sup_{|t|\le T}\left(|\alpha(t)^2+\dot{\theta}(t)|+|\dot{\alpha}(t)|\right)\lesssim \varepsilon^2.
\end{equation*}
\end{proposition}

\begin{remark}\label{rem:dottheta}
It is easily to understand that the assumptions
$$
\sup_{|t|\ll \log(1/\varepsilon)}|\alpha(t)-1|\lesssim \varepsilon,~\sup_{|t|\ll \log(1/\varepsilon)}|\alpha(t)^2+\dot{\theta}(t)|\lesssim \varepsilon^2,~\sup_{|t|\ll \log(1/\varepsilon)}|\dot{\alpha}(t)|\lesssim \varepsilon^2
$$
imply that
\begin{equation*}
\begin{split}
1+\dot{\theta}(t) & = \alpha(t)^2+\dot{\theta}(t)+(1-\alpha(0)^2)-2\int_0^t\alpha(t')\dot{\alpha}(t')\,dt'\\
& = O(\varepsilon^2)+O(\varepsilon)+O(\varepsilon^2 t)=O(\varepsilon)
\end{split}
\end{equation*}
for $|t|\ll \log(1/\varepsilon)$.
\end{remark}

\begin{remark}
The size assumption on the data in \eqref{eq:data} leads $\|v_e(0)\|_{H^1}+\|v_o(0)\|_{H^1}\lesssim \varepsilon$.
From \eqref{eq:a_e} and \eqref{eq:b_e} along with Lemma $\ref{lem:c^1map}$, it follows
\begin{equation}\label{eq:e^2}
|a_e(0)|+|b_e(0)|\lesssim \varepsilon^2.
\end{equation}
In this sense, the condition \eqref{eq:data} contains an auxiliary hypothesis.
\end{remark}

\begin{proof}
The proof follows from the boot strap argument in conjunction with the a priori estimates in \eqref{eq:a_e}, \eqref{eq:a_ed},  \eqref{eq:b_e}, \eqref{eq:b_ed}, \eqref{eq:a_o}, \eqref{eq:b_o}, \eqref{eq:theta}, \eqref{eq:alpha}, \eqref{eq:eta_e} and \eqref{eq:eta_o}.

Let us introduce the following notation
\begin{equation*}
\Lambda(t)= |a_e(t)|+|b_e(t)|+|a_o(t)|+|b_o(t)|+\|\eta_e(t)\|_{H^1}+\|\eta_o(t)\|_{H^1}+|\alpha(t)-1|
\end{equation*}
and
\begin{equation*}
\Xi(t)=|\alpha(t)^2+\dot{\theta}(t)|+|\dot{\alpha}(t)|.
\end{equation*}
We note that by \eqref{eq:theta}, \eqref{eq:alpha}, \eqref{eq:e^2} and $|\alpha(0)-1|\lesssim \varepsilon$, we have
\begin{equation*}
\Lambda(0)\le c\varepsilon,\quad 
\Xi(0)\le c\varepsilon^2.
\end{equation*}
Suppose the bootstrap assumption as follows
\begin{equation}\label{eq:bootstrap}
\sup_{|t|\le \delta}\Lambda(t)\le 2c\varepsilon,\quad 
\sup_{|t|\le \delta}\Xi(t)\le 2c\varepsilon^2
\end{equation}
for some $\delta>0$.

From \eqref{eq:a_e} and \eqref{eq:b_e}, we obtain
\begin{equation}\label{eq:a1}
|a_e(t)|+|b_e(t)|\lesssim \|v_e(t)\|_{L^2}|\alpha(t)-1|\lesssim \varepsilon^2.
\end{equation}
Thus by \eqref{eq:theta}, \eqref{eq:alpha} and \eqref{eq:a1}, it follows that
\begin{equation}\label{eq:a2}
\begin{split}
&|\alpha(t)^2+\dot{\theta}(t)|\\
& \lesssim  |b_e(t)|+\left(|\dot{\alpha}(t)|+ |\alpha(t)-1|\right)\|v_e(t)\|_{L^2}+\|v_e(t)\|_{H^1}^2+\|v_o(t)\|_{H^1}^2\\
& \lesssim    \varepsilon^2,
\end{split}
\end{equation}
and
\begin{equation}\label{eq:a3}
|\dot{\alpha}(t)|  \lesssim |\alpha(t)-1|\|v_e(t)\|_{L^2}+|1+\dot{\theta}(t)|\|v_e(t)\|_{H^1}+\|v_e(t)\|_{H^1}^2+\|v_o(t)\|_{H^1}^2\lesssim   \varepsilon^2,
\end{equation}
which yields
\begin{equation}\label{eq:a31}
|\alpha(t)-1|\le |\alpha(0)-1|+\left|\int_0^t\dot{\alpha}(t')\,dt'\right|\lesssim \varepsilon+\varepsilon^2|t|.
\end{equation} 

In a similar way, \eqref{eq:a_ed}, \eqref{eq:b_ed} and \eqref{eq:a2} yield
\begin{equation*}
\begin{split}
&|\dot{a}_e(t)| \\
& \lesssim  |\alpha(t)^2+\dot{\theta}(t)|+\left(|\alpha(t)-1|+|1+\dot{\theta}(t)|\right)\|v_e(t)\|_{L^2}+\|v_e(t)\|_{H^1}^2+\|v_o(t)\|_{H^1}^2\\
& \lesssim  \varepsilon^2
\end{split}
\end{equation*}
and
\begin{equation*}
\begin{split}
|\dot{b}_e(t)|
& \lesssim |\dot{\alpha}(t)|+\left(|1+\dot{\theta}(t)|+|\alpha(t)-1|\right)|\|v_e(t)\|_{L^2}+\|v_e(t)\|_{H^1}^2+\|v_o(t)\|_{H^1}^2\\
&\lesssim   \varepsilon^2.
\end{split}
\end{equation*}

Also from \eqref{eq:a_o} and \eqref{eq:b_o}, it follows that
\begin{equation*}
\begin{split}
& |\dot{a}_o(t)|\\
& \lesssim  |b_o(t)|+\left(|1+\dot{\theta}(t)|+|\alpha(t)-1|\right)\|v_o(t)\|_{L^2}+\|v_e(t)\|_{H^1}^2+\|v_o(t)\|_{H^1}^2\\
& \lesssim   |b_o(t)|+\varepsilon^2
\end{split}
\end{equation*}
and
\begin{equation*}
\begin{split}
 |\dot{b}_o(t)|
& \lesssim  \left(|1+\dot{\theta}(t)|+|\alpha(t)-1|\right)\|v_o(t)\|_{L^2}+\|v_e(t)\|_{H^1}^2+\|v_o(t)\|_{H^1}^2\\
& \lesssim   \varepsilon^2,
\end{split}
\end{equation*}
respectively.
If we take the integral of the above functions, we obtain
\begin{equation}\label{eq:a4}
|b_o(t)|\lesssim  |b_o(0)|+|t|\varepsilon^2
\end{equation}
and thus
\begin{equation}\label{eq:a5}
|a_o(t)|\lesssim |a_o(0)|+\left(|b_o(0)|+\varepsilon^2\right)|t|+\varepsilon^2t^2.
\end{equation}

Next, it is clear from \eqref{eq:eta_e} and \eqref{eq:eta_o} along with the similar argument to above that
\begin{equation*}
\begin{split}
\frac{d}{dt}\|\eta_e(t)\|_{L^2} & \lesssim \|\eta_e(t)\|_{L^2}+|\dot{a}_e(t)|+|\dot{b}_e(t)|+|b_e(t)|+\varepsilon^2\\
& \lesssim  \|\eta_e(t)\|_{L^2}+\varepsilon^2
\end{split}
\end{equation*}
and
\begin{equation*}
\begin{split}
\frac{d}{dt}\|\eta_o(t)\|_{L^2} & \lesssim \|\eta_o(t)\|_{L^2}+|\dot{a}_o(t)|+|\dot{b}_e(t)|+|b_o(t)|+\varepsilon^2\\
& \lesssim  \|\eta_o(t)\|_{L^2}+|b_o(0)|+\varepsilon^2+\varepsilon^2|t|.
\end{split}
\end{equation*}
Hence by Gronwall's inequality, we see that
\begin{equation}\label{eq:a5}
\|\eta_e(t)\|_{L^2}\lesssim \|\eta_e(0)\|_{L^2}e^{c|t|}+\varepsilon^2(e^{c|t|}-1)\le \left(\|\eta_e(0)\|_{L^2}+\varepsilon^2\right)e^{c|t|}
\end{equation}
and
\begin{equation}\label{eq:a51}
\begin{split}
\|\eta_o(t)\|_{L^2}& \lesssim \|\eta_o(0)\|_{L^2}e^{c|t|}+\left(|b_o(0)|+\varepsilon^2\right)(e^{c|t|}-1)+\varepsilon^2(e^{c|t|}-1-c|t|)\\
& \le \left( \|\eta_o(0)\|_{L^2}+|b_o(0)|+\varepsilon^2\right)e^{c|t|}.
\end{split}
\end{equation}

Similar to above, we obtain
\begin{equation*}
\begin{split}
\frac{d}{dt}\|\partial_x\eta_e(t)\|_{L^2} & \lesssim \|\eta_e(t)\|_{H^1}+|\dot{a}_e(t)|+|\dot{b}_e(t)|+|b_e(t)|+\varepsilon^2\\
& \lesssim  \|\partial_x\eta_e(t)\|_{L^2}+(\|\eta_e(0)\|_{L^2}+\varepsilon^2)e^{c|t|}
\end{split}
\end{equation*}
and
\begin{equation*}
\begin{split}
\frac{d}{dt}\|\partial_x\eta_o(t)\|_{L^2} & \lesssim \|\eta_o(t)\|_{H^1}+|\dot{a}_e(t)|+|\dot{b}_e(t)|+|b_o(t)|+\varepsilon^2\\
& \lesssim  \|\partial_x\eta_o(t)\|_{L^2}+ (\|\eta_o(0)\|_{L^2}+|b_o(0)|+\varepsilon^2)e^{c|t|},
\end{split}
\end{equation*}
which imply
\begin{equation}\label{eq:a6}
\begin{split}
\|\partial_x\eta_e(t)\|_{L^2} &\lesssim \|\partial_x\eta_e(0)\|_{L^2}e^{c|t|}+(\|\eta_e(0)\|_{L^2}+\varepsilon^2)e^{c|t|}|t|
\end{split}
\end{equation}
and
\begin{equation}\label{eq:a7}
\begin{split}
\|\partial_x\eta_o(t)\|_{L^2}& \lesssim \|\partial_x\eta_o(0)\|_{L^2}e^{c|t|}+(\|\eta_o(0)\|_{L^2}+|b_o(0)|+\varepsilon^2)e^{c|t|}|t|,
\end{split}
\end{equation}
respectively.

Now, we may take the data $a_e(0),~b_e(0),~a_o(0),~b_o(0),~\eta_e(0),~\eta_o(0)$ such as \eqref{eq:data}.
Fix $T_{\varepsilon}\ll \log(1/\varepsilon)$.
A bootstrap argument in conjunction with \eqref{eq:a1}, \eqref{eq:a2}, \eqref{eq:a3}, \eqref{eq:a4}, \eqref{eq:a5}, \eqref{eq:a6} and \eqref{eq:a7} yields that the estimate \eqref{eq:bootstrap} holds as long as $\delta\le T_{\varepsilon}$, then
\begin{equation*}
\sup_{|t|\le T_{\varepsilon}}\Lambda(t)\le 2c\varepsilon,\quad
\sup_{|t|\le T_{\varepsilon}}\Xi(t)\le 2c\varepsilon^2,
\end{equation*}
which is the desired result.
\end{proof}

The following lemma is used to justify the sufficient condition on the size of data in \eqref{eq:data}.

\begin{lemma}\label{lem:compare}
Let $0<\varepsilon\ll 1,~\theta\in\mathbb{R},~\alpha\in\mathbb{R}$ with $|\theta|\lesssim \varepsilon,~|\alpha-1|\lesssim \varepsilon$.
Let $a,~b,~a',~b'$ be real numbers, and let $\eta,~\eta'$ be even, odd functions in $H^1$, respectively with
\begin{equation*}
\begin{gathered}
\langle \eta\mid Q\rangle=\langle i\eta\mid Q'\rangle=\langle \eta'\mid \partial_xQ\rangle=\langle i\eta'\mid xQ\rangle=0.
\end{gathered}
\end{equation*}
If 
\begin{equation*}
aiQ+bQ'+\eta+a'i\partial_xQ+b'xQ+\eta'=Q-e^{i\theta}Q_{\alpha}+O_{H^1}(\varepsilon^2),
\end{equation*}
then
\begin{equation*}
\begin{gathered}
|a|+|b|\lesssim \varepsilon,\\
|a'|+|b'|+\|\eta\|_{H^1}+\|\eta'\|_{H^1}\lesssim \varepsilon^2.
\end{gathered}
\end{equation*}
Here $O_{H^1}(\epsilon)$ denotes a $H^1$-function such that $\|O_{H^1}(\epsilon)\|_{H^1}\lesssim \epsilon$. 
\end{lemma}

\begin{proof}
Since $a'i\partial_xQ+b'xQ+\eta'=O_{H^1}(\varepsilon^2)$, we easily see that
\begin{equation*}
|a'|+|b'|+\|\eta'\|_{H^1}\lesssim \varepsilon^2.
\end{equation*}

On the other hand, the hypothesis guarantees that
\begin{equation*}
\begin{split}
aiQ+bQ'+\eta & = Q-e^{i\theta}Q_{\alpha}+O_{H^1}(\varepsilon^2)\\
& = Q-Q_{\alpha}-i\sin\theta Q_{\alpha}+O_{H^1}(\varepsilon^2).
\end{split}
\end{equation*}
Taking the semi-inner product with $iQ'$, we have
\begin{equation*}
-M(Q)a= \sin\theta\langle Q_{\alpha}\mid Q'\rangle+O(\varepsilon^2)=O(\varepsilon).
\end{equation*}
Similarly, by the semi-inner product with $Q$, it follows
\begin{equation*}
M(Q)b=\langle Q\mid Q\rangle -\langle Q_{\alpha}\mid Q\rangle+O(\varepsilon^2)=\langle Q-Q_{\alpha}\mid Q\rangle +O(\varepsilon^2)=O(\varepsilon).
\end{equation*}
For $\eta$, we have that
\begin{equation*}
\begin{split}
\eta= & Q-Q_{\alpha}-\frac{1}{M(Q)}\langle Q-Q_{\alpha}\mid Q\rangle Q'+O_{H^1}(\varepsilon^2)\\
= & \frac{1}{M(Q)}\left(\langle Q'\mid Q\rangle (Q-Q_{\alpha})-\langle Q-Q_{\alpha}\mid Q\rangle Q'\right)+O_{H^1}(\varepsilon^2)\\
 = & \frac{1}{M(Q)}\int_1^{\alpha}d\beta\int_1^{\beta}\frac{\partial}{\partial\gamma}\left(\langle Q'_{\beta-\gamma+1}\mid Q\rangle Q'_{\gamma}\right)\,d\gamma+O_{H^1}(\varepsilon^2)\\
 = & O_{H^1}(\varepsilon^2),
\end{split}
\end{equation*}
which completes the proof.
\end{proof}

Now we are in position to prove Theorem \ref{thm:main}.

\begin{proof}[Proof of Theorem $\ref{thm:main}$]
Consider the initial data $u_0\in\mathcal{S}_{\varepsilon}$ such that 
$$
\left\|u_0-\left(Q+A_e iQ+B_eQ'+A_oi\partial_xQ\right)\right\|_{H^1}\le \varepsilon^2,
$$
where $A_e,~B_e,~A_o\in\mathbb{R}$ such that $|A_e|+|B_e|\le |A_o|=\varepsilon$.
By the symplectic decomposition, we have the form
$$
u_0=Q+A'_e iQ+B'_eQ'+A'_oi\partial_xQ+B'_oxQ+g,
$$
where $A'_e,~B'_e,~A'_o,~B'_o\in\mathbb{R}$ and $g\in H^1$ such that
$$
\Re\int_{\mathbb{R}}g(x)Q(x)\,dx=0,\quad \Im\int_{\mathbb{R}}g(x)Q'(x)\,dx=0,
$$
$$
\Re\int_{\mathbb{R}}g(x)\partial_xQ(x)\,dx=0,\quad \Im\int_{\mathbb{R}}g(x)xQ(x)\,dx=0.
$$
By Lemma \ref{lem:compare}, it follows that
$$
|A'_e|+|B'_e|\lesssim\varepsilon,\quad |A'_o|+|B'_o|+\|g\|_{H^1}\lesssim\varepsilon^2.
$$
We find data $v_0(x)=v(0,x)$ required in Theorem \ref{thm:main} satisfying 
$$
u_0=Q+A'_e iQ+B'_eQ'+A'_oi\partial_xQ+B'_oxQ+g=e^{i\theta(0)}\left(Q_{\alpha(0)}+v_0\right),
$$
where $\theta(0)$ and $\alpha(0)$ are obtained in Lemma \ref{lem:c^1map}.

By Lemma \ref{lem:c^1map}, we have $|\alpha(0)-1|\lesssim\varepsilon$ and $\|v_0\|_{H^1}\lesssim \varepsilon$.
Moreover, $Q-e^{i\theta(0)}Q_{\alpha(0)}=O_{H^1}(\varepsilon)$ implies that there exists some $n\in\mathbb{Z}$ such that $|\theta(0)-2n\pi|\lesssim \varepsilon$.

We revisit the symplectic decomposition to have
$$
e^{i\theta(0)}v_0=a'_e iQ+b'_eQ'+a'_oi\partial_xQ+b'_oxQ+g_1,
$$
where $a'_e,~b'_e,~a'_o,~b'_o\in\mathbb{R}$ and $g_1\in H^1$ such that
$$
\Re\int_{\mathbb{R}}g_1(x)Q(x)\,dx=0,\quad \Im\int_{\mathbb{R}}g_1(x)Q'(x)\,dx=0,
$$
$$
\Re\int_{\mathbb{R}}g_1(x)\partial_xQ(x)\,dx=0,\quad \Im\int_{\mathbb{R}}g_1(x)xQ(x)\,dx=0.
$$
Again, by Lemma \ref{lem:compare}, it follows that
$$
|A'_e-a'_e|+|B'_e-b'_e|\lesssim\varepsilon,\quad |A'_o-a'_o|+|B'_o-b'_o|+\|g-g_1\|_{H^1}\lesssim\varepsilon^2,
$$
and then
\begin{equation}\label{eq:lemma6}
|a'_e|+|b'_e|\lesssim\varepsilon,\quad|a'_o|+ |b'_o|+\|g_1\|_{H^1}\lesssim\varepsilon^2.
\end{equation}
We use $e^{-i\theta(0)}=1+O(\varepsilon)$ to have
\begin{equation*}
\begin{split}
v_0 & =e^{-i\theta(0)}\left(a'_e iQ+b'_eQ'+a'_oi\partial_xQ+b'_oxQ+g_1\right)\\
& = a'_e iQ+b'_eQ'+a'_oi\partial_xQ+b'_oxQ+O_{H^1}(\varepsilon^2).
\end{split}
\end{equation*}

One applies the symplectic decomposition to $v_0$ as follows
$$
v_0=a_e iQ+b_eQ'+a_oi\partial_xQ+b_oxQ+g_2,
$$
where $a_e,~b_e,~a_o,~b_o\in\mathbb{R}$ and $g_2\in H^1$ such that
$$
\Re\int_{\mathbb{R}}g_2(x)Q(x)\,dx=0,\quad \Im\int_{\mathbb{R}}g_2(x)Q'(x)\,dx=0,
$$
$$
\Re\int_{\mathbb{R}}g_2(x)\partial_xQ(x)\,dx=0,\quad \Im\int_{\mathbb{R}}g_2(x)xQ(x)\,dx=0.
$$
By Lemma \ref{lem:compare} and \eqref{eq:lemma6}, we have
$$
|a_e|+|b_e|\lesssim\varepsilon,\quad |a_o|+ |b_o|+\|g_2\|_{H^1}\lesssim\varepsilon^2,
$$
which corresponds to the hypotheses \eqref{eq:data} required in Proposition \ref{prop:apriori}. 
Hence Proposition \ref{prop:apriori} guarantees the existence of solution $u(t)$ to \eqref{eq:nnls} on the time interval $[-T_{\varepsilon},T_{\varepsilon}]$ provided $T_{\varepsilon}\ll \log (1/\varepsilon)$.

Hence we complete the proof of Theorem \ref{thm:main}.
\end{proof}

\begin{remark}
The equation \eqref{eq:nnls} possesses soliton solutions
$$
u_{\alpha,\beta,\gamma}(t,x)=\frac{\sqrt{2}(\alpha+\beta)}{e^{i\alpha^2 t+\alpha(x-i\gamma)}+e^{i\beta^2 t-\beta(x-i\gamma)}},
$$
where three-parameter $\alpha>0,~\beta>0,~\gamma\in\mathbb{R}$.
If $u_0=e^{i\theta(0)}(Q_{\alpha(0)}+v_0)\in\mathcal{S}_{\varepsilon}$ is given by
$$
u_0=Q+a_oi\partial_xQ+O_{H^1}(\varepsilon^2),\quad |a_o|=\varepsilon,
$$
it is reasonable to suppose
$$
\inf_{\gamma,\theta\in\mathbb{R},~\alpha>0}\|u_0(\cdot)-e^{i\theta}u_{\alpha,\alpha,\gamma}(0,\cdot)\|_{H^1}\sim \|u_0(\cdot)-e^{i\theta'}u_{\alpha,\alpha,\gamma'}(0,\cdot)\|_{H^1}\lesssim \varepsilon^2
$$
for appropriate $\theta',~\gamma'$ such that $\gamma'=a_o+O(\varepsilon^2)=O(\varepsilon)$.
By using the relative compactness of the potential of the linearized operator, the spectral mapping theorem and a usual estimate of the Duhamel formula of solutions, we may have
\begin{equation*}
\begin{split}
\inf_{\theta\in\mathbb{R}}\|u(t)-e^{i\theta}Q\|_{H^1} & \lesssim |\alpha-1|+|\gamma'|+\|e^{\mathcal{H}t}\|_{OP}\|u_0(\cdot)-e^{i\theta'}u_{\alpha,\alpha,\gamma'}(0,\cdot)\|_{H^1}\\
&\quad +\int_0^t\|e^{\mathcal{H}(t-t')}\|_{OP}\|\text{(nonlinear term(s))}\|_{H^1}\,dt'\\
& \lesssim \varepsilon+e^{ct}\varepsilon^2+\int_0^te^{c(t-t')}\|\text{(nonlinear term(s))}\|_{H^1}\,dt',
\end{split}
\end{equation*}
where $\mathcal{H}$ is a linearized operator and $\|\cdot\|_{OP}$ denotes the operator norm.
If there has the a priori estimate $\|\text{(nonlinear term(s))}\|_{H^1}\lesssim \varepsilon^2$, we can show Theorem \ref{thm:main} by the linear stability of \eqref{eq:nnls} around $u_{1,1,0}$.
However it is not clear that this is a superior approach, since it needs to compute the time evolution on $\gamma=\gamma(t)$ as well as $\theta,~\alpha$.
\end{remark}

\section*{Appendices: Estimates \eqref{eq:blowup}, \eqref{eq:blowup2}, \eqref{eq:data-ex} and \eqref{eq:data-diff}  }\label{secA1}

This appendix presents the formulation computing the estimates \eqref{eq:blowup}, \eqref{eq:blowup2}, \eqref{eq:data-ex} and \eqref{eq:data-diff}.
For simplicity, we only consider the case $k=0$.

First we prove \eqref{eq:blowup}.
By a straightforward computation we see that $\|u_{\alpha,\beta}(t)\|_{L^2}^2$ is
$$
\int_0^{\infty}\frac{dx}{e^{2\alpha x}|e^{i(\alpha^2-\beta^2)t}+e^{-(\alpha+\beta)x}|^2}+\int_{-\infty}^0\frac{dx}{e^{-2\beta x}|e^{i(\alpha^2-\beta^2)t+(\alpha+\beta)x}+1|^2}.
$$
For small $0<\delta\ll 1$, we have
$$
\|u_{\alpha,\beta}(t,\cdot)\|_{L^2}^2\gtrsim \int_0^{\delta}\frac{dx}{|e^{i(\alpha^2-\beta^2)t}+e^{-(\alpha+\beta)x}|^2}+\int_{-\delta}^0\frac{dx}{|e^{i(\alpha^2-\beta^2)t+(\alpha+\beta)x}+1|^2}.
$$
Then by Fatou's lemma, one has
\begin{equation*}
\begin{split}
\liminf_{t\to \pm\pi/(\alpha^2-\beta^2)}\|u_{\alpha,\beta}(t)\|_{L^2}^2
& \gtrsim   \int_0^{\delta}\liminf_{t\to \pm\pi/(\alpha^2-\beta^2)}\frac{dx}{|e^{i(\alpha^2-\beta^2)t}+e^{-(\alpha+\beta)x}|^2}\\
& \quad +\int_{-\delta}^0\liminf_{t\to \pi/(\alpha^2-\beta^2)}\frac{dx}{|e^{i(\alpha^2-\beta^2)t+(\alpha+\beta)x}+1|^2}\\
& =  \int_0^{\delta}\frac{dx}{|-1+e^{-(\alpha+\beta)x}|^2}+\int_{-\delta}^0\frac{dx}{|-e^{(\alpha+\beta)x}+1|^2},
\end{split}
\end{equation*}
which diverges to infinity.
Hence we conclude that $\|u_{\alpha,\beta}(t)\|_{L^2}$ also diverges to infinity as $t$ going to $\pm\pi/(\alpha^2-\beta^2)$.

Second, we prove \eqref{eq:blowup2}.
When $|x|\ll 1$ and $|t||\alpha-\beta|\ll 1\lesssim |t|$, one has
\begin{equation*}
\begin{split}
& |u_{\alpha,\alpha}(t,x)-e^{i\theta}u_{\alpha,\beta}(t,x)| \\
& \sim \left|\frac{2\alpha}{2+O(x^2)}-\frac{(\alpha+\beta)e^{i\theta}}{2+(\alpha-\beta)x+i(\beta^2-\alpha^2)t+O\left(x^2+((\alpha^2-\beta^2)t)^2\right)}\right|.
\end{split}
\end{equation*}
If $|\theta-2n\pi|\gtrsim 1$ for some $n\in \mathbb{Z}$, then we easily see that
$$
 |u_{\alpha,\alpha}(t,x)-e^{i\theta}u_{\alpha,\beta}(t,x)|\gtrsim 1,\quad \text{for $|x|\ll 1,~|t||\alpha-\beta|\ll 1$},
$$
then $\|u_{\alpha,\alpha}(t,\cdot)-e^{i\theta}u_{\alpha,\beta}(t,\cdot)\|_{L^2}\gtrsim 1$.
On the other hand, if $|\theta-2n\pi|\ll 1$ for some $n\in \mathbb{Z}$, we have that
\begin{equation*}
\begin{split}
& |u_{\alpha,\alpha}(t,x)-e^{i\theta}u_{\alpha,\beta}(t,x)|\\
& \sim \left|2\alpha\left(e^{\alpha x}+e^{i(\beta^2-\alpha^2)t-\beta x}-e^{i\theta}(e^{\alpha x}+e^{-\alpha x})\right)+(\alpha-\beta)e^{i\theta}(e^{\alpha x}+e^{-\alpha x})\right|\\
& \sim \left|2\alpha\left(2+(\alpha-\beta)x+i(\beta^2-\alpha^2)t-2(1+O(\theta-2n\pi))\right)\right.\\
&\qquad \left.+2(\alpha-\beta)(1+O(\theta-2n\pi))+O(x^2+(\alpha-\beta)^2t^2)\right| \\
 & \gtrsim |\alpha-\beta|| t|
\end{split}
\end{equation*}
for  $|x|\ll 1,~|t||\alpha-\beta|\ll 1\lesssim |t|$.
Taking $L^2$ norm yields the estimate \eqref{eq:blowup2}.

Thirdly, we prove \eqref{eq:data-ex} by the contradiction.
If $u_{1,\beta}(0,\cdot)\in \mathcal{S}_{\varepsilon}$ for some $0<\varepsilon\ll 1$, we have
$$
\| u_{1,\beta}(0,\cdot)-\left(Q+ia_eQ+b_eQ'+a_oi\partial_x Q\right)\|_{H^1}\lesssim \varepsilon^2.
$$
Since $u_{1,\beta}(0,\cdot),~Q,~Q'$ are real-valued functions, we have $|a_e|+|a_o|\lesssim \varepsilon^2$, which contradicts with $|a_o|=\varepsilon$.

Finally, we prove \eqref{eq:data-diff}.
The proof of Theorem \ref{thm:main} is valid.
We repeat the argument to get \eqref{eq:data-diff}.
Suppose the following:
$$
f=Q+A_eiQ+B_eQ'+A_oi\partial_xQ+O_{H^1}(\varepsilon^2)=e^{i\theta}Q+v,
$$
$$
A_e,~B_e,~A_o\in\mathbb{R},\quad |A_e|+|B_e|\le |A_o|=\varepsilon,
$$
$$
d(f,Q)\sim \|f-e^{i\theta}Q\|_{H^1}=\|v\|_{H^1}\lesssim \varepsilon.
$$
We use the symplectic decomposition to have
$$
v=a_e iQ+b_eQ'+a_oi\partial_xQ+b_oxQ+g(x),
$$
where $a_e,~b_e,~a_o,~b_o\in\mathbb{R}$ and $g\in H^1$ such that
$$
\Re\int_{\mathbb{R}}g(x)Q(x)\,dx=0,\quad \Im\int_{\mathbb{R}}g(x)Q'(x)\,dx=0,
$$
$$
\Re\int_{\mathbb{R}}g(x)\partial_xQ(x)\,dx=0,\quad \Im\int_{\mathbb{R}}g(x)xQ(x)\,dx=0.
$$
We have that $|\theta-2n\pi|\lesssim \varepsilon$ for some $n\in\mathbb{Z}$ and
$$
|A_e-a_e|+|B_e-b_e|\lesssim \varepsilon,\quad |A_o-a_o|+|b_o|+\|g\|_{H^1}\lesssim \varepsilon^2,
$$
which will be operated in the same way as before.
Then $\|v\|_{H^1}\gtrsim |a_o|\sim |A_o|=\varepsilon$, which concludes the proof of \eqref{eq:data-diff}.


\begin{thebibliography}{00}

\bibitem{ablowitz1}
M. J. Ablowitz, D. J. Kaup, A. C. Newell and H. Segur,
The inverse scattering transform-Fourier analysis for nonlinear problems,
Stud. Appl. Math.,
\textbf{53} (1974), 249--315.

\bibitem{ablowitz2}
M. J. Ablowitz and Z. H. Musslimani,
Integrable nonlocal nonlinear Schr\"odinger equation,
Phys. Rev. Lett.,
\textbf{110} (2013), 064105.

\bibitem{ablowitz3}
M. J. Ablowitz and Z. H. Musslimani,
Inverse scattering transform for the integrable nonlocal nonlinear Schr\"odinger equation,
Nonlineality,
\textbf{29} (2016), 915--946.

\bibitem{chang}
S.-M. Chang, S. Gustafson, K. Nakanishi and T.-P. Tsai,
Spectra of linearized operator for NLS solitary waves,
SIAM J. Math. Anal.,
\textbf{39} (2007), 1070--1111. 

\bibitem{gadzhimuradov}
T. A. Gadzhimuradov and A. M. Agalarov,
Towards a gauge-equivalent magnetic structure of the nonlocal nonlinear Schr\"odinger equation,
Phys. Rev. A,
\textbf{93} (2016), 062124. 

\bibitem{genoud}
F. Genoud,
Instability of an integrable nonlocal NLS,
C. R. Acad. Sci. Paris, Ser. I,
\textbf{355} (2017), 299--303.

\bibitem{cazenave}
T. Cazenave and F. B. Weissler,
The Cauchy problem for the critical nonlinear Schr\"odinger equation in $H^s$,
Nonlinear Anal. TMA,
\textbf{14} (1990), 807--836.

\bibitem{chen}
J. Chen, Y. Liu and B. Wang,
Global Cauchy problems for the nonlocal (derivative) NLS in $E_{\sigma}^s$,
J. Diff. Equations,
\textbf{344} (2023), 767--806.

\bibitem{gurses}
M. G\"ureses and A. Pekcan,
Nonlocal nonlinear Schr\"odinger equations and their soliton solutions,
J. Math. Phys.,
\textbf{59} (2018), 051501.

\bibitem{krieger}
J. Krieger and W. Schlag,
Stable manifolds for all monic supercritical focusing nonlinear Schr\"odinger equations in one dimension,
J. Amer. Math. Soc.,
\textbf{19} (2006), 815--920.

\bibitem{okamoto}
M. Okamoto and K. Uriya,
Final state problem for the nonlocal nonlinear Schr\"odinger equation with dissipative nonlinearity,
DEA, Differ. Equ. Appl.,
\textbf{4} (2019), 481--494.

\bibitem{martel}
Y. Martel and F. Merle,
Instability of solitons for the critical generalized Koeteweg-de Vries equation,
Geom. Func. Anal.,
\textbf{11} (2001), 74--123.

\bibitem{nakanishi}
K. Nakanishi and W. Schlag,
Global dynamics above the ground state energy for the cubic NLS equation in 3D,
Calc. Var.,
\textbf{44} (2012), 1--45.

\bibitem{rybalko}
Y. Rybalko and D. Shepelsky,
Global conservative solutions of the nonlocal NLS equation beyond blow-up,
Discrete Contin. Dyn. Syst.,
\textbf{43} (2023), 860--894.

\bibitem{tsutsumi}
Y. Tsutsumi,
$L^2$-solutions for nonlinear Schr\"odinger equations and nonlinear groups,
Funkc. Ekvacioj,
\textbf{30} (1987), 115--125.

\bibitem{weinstein1}
M. I. Weinstein,
Modulational stability of ground states of nonlinear Schr\"odinger equations,
SIAM J. Math. Anal.,
\textbf{16} (1985), 472--491.

\bibitem{yang}
J. Yang,
Physically significant nonlocal nonlinear Schr\"odinger equation and its soliton solutions,
Phys. Rev. E,
\textbf{98}, 042202.

\bibitem{zakharov}
V. E. Zakharov and A. B. Shabat,
Exact theory of two-dimensional self-focusing and one-dimensional self-modulation of waves in nonlinear media,
Sov. Phys. JETP,
\textbf{34} (1972), 62--69.

\bibitem{zhao}
Y. Zhao and E. Fan,
Existence of global solutions to the nonlocal Schr\"odinger equation on the line,
Stud. Appl. Math.,
\textbf{152} (2024), 111--146.

\end{thebibliography}
\end{document}